\documentclass[11pt,a4paper,reqno]{amsart}
\usepackage{amssymb}
\usepackage{amscd}
\usepackage{amsmath}
\usepackage{amsfonts}

\usepackage{cite}

\theoremstyle{plain}
\newtheorem{theorem}{Theorem}[section]

\newtheorem{lemma}[theorem]{Lemma}

\theoremstyle{definition}
\newtheorem{definition}[theorem]{Definition}

\def\Shape(#1){\operatorname{Shape}(#1)}

\newcommand{\magma}{{\sc Magma}}
\newcommand{\Z}{\mbox{\bf Z}}
\newcommand{\Q}{\mbox{\bf Q}}
\newcommand{\F}{\mbox{\bf F}}
\newcommand{\R}{\mbox{\bf R}}
\newcommand{\C}{\mbox{\bf C}}
\newcommand{\Prj}{\mbox{\bf P}}
\newcommand{\XN}{X_0(N)}
\newcommand{\X}{X_0(108)}
\newcommand{\kq}{k(108)}
\newcommand{\kd}{k'(108)}
\newcommand{\BN}{B_0(N)}
\newcommand{\AN}{A_0(N)}
\newcommand{\JN}{J_0(N)}
\newcommand{\autC}{\mbox{Aut}_{\scriptstyle{\bf C}}(\XN)}
\newcommand{\rtm}{\sqrt{-3}}

\begin{document}
\title{A New Automorphism of $X_0(108)$}
\author{Michael Harrison} 
\address{School of Mathematics and Statistics F07, University of Sydney, NSW 2006, Australia}
\begin{abstract}
Let $\XN$ denote the modular curve classifying elliptic curves
  with a cyclic $N$-isogeny,
$\AN$ its group of algebraic automorphisms and $\BN$ the subgroup of automorphisms 
coming from matrices acting on the upper half-plane. In a well-known
paper, Kenku and Momose showed that $\AN$ and $\BN$ are equal (all automorphisms
come from matrix action) when $\XN$ has genus $\ge 2$, except for $N = 37$ and $63$.

However, there is a mistake in their analysis of the $N = 108$ case. In the
style of Kenku and Momose, we show that
$B_0(108)$ is of index 2 in $A_0(108)$ and construct an explicit new automorphism
of order 2 on a canonical model of $\X$.
\end{abstract}
\maketitle

\section{Introduction}
For a positive integer $N$, the modular curve $\XN$ parametrises elliptic curves
with a cyclic $N$-isogeny. Over $\C$, it is isomorphic, as a Riemann surface, to the 
quotient of the extended upper half-plane (${\tau \in \C : \Im(\tau) > 0} \cup \Q
\cup {i\infty}$) by the subgroup $\Gamma_0(N)$ of $SL_2(\Z)$
  consisting of determinant 1,
integral matrices $\left({a \atop c}{b\atop d}\right)$ with
$c \equiv 0$ mod $N$, acting as $\tau \mapsto (a\tau + b)/(c\tau + d)$
(\cite{Sh71} or \cite{Mi89}).

$\XN$ has a natural structure of an algebraic curve over $\Q$, defined in
\cite[Ch. 7]{Sh71} or more technically, as the generic fibre of the compactification 
of a modular scheme over $\Z$ \cite{KM85}.

The normaliser $Nm(N)$ of $\Gamma_0(N)$ in $SL_2(\R)$ acts on the extended upper
half-plane and leads to a finite subgroup $\BN$ of $\AN \stackrel{\mathrm{def}}{=} 
\autC$ isomorphic to $Nm(N)/\Gamma_0(N)$. The group-theoretic structure
of $\BN$ is given in \cite{Bar08}. The natural question, when the genus $g_N$ of
$\XN \ge 2$ and so $\AN$ is finite, is whether $\BN$ is all of $\AN$. For
$N = 37$ this was famously known not to be the case: $X_0(37)$ is of genus
2 and is thus hyperelliptic, but the only non-trivial element of $B_0(37)$ is
not a hyperelliptic involution. Ogg showed that this is the only case with
$\AN$ larger than $\BN$ when $N$ is squarefree \cite{Ogg77}.

Using deeper properties
of the minimal models of $\XN$ and its Jacobian $\JN$ and some further techniques, 
Kenku and Momose extended the analysis to all $N$ (with $g_N \ge 2$) in \cite{KM88},
claiming that $\AN = \BN$ except for $N = 37$ and possibly $N = 63$, and that
the index of $\BN$ in $\AN$ is 1 or 2 in the latter case. Subsequently, Elkies
showed that $N = 63$ is indeed an exceptional case and gave an elegant construction
of an additional automorphism \cite{Elk90}. Kenku and Momose eliminate almost all
$N$ by a combination of arguments that lead to only seven values (including 63 and
108) for which case-by-case detailed analysis is required.

For $N = 108$, a 
hypothetical automorphism $u$ not in $\BN$ is considered and is used to construct
a non-trivial automorphism $\gamma$ in $\BN$ with various properties. $\JN$
decomposes (up to isogeny) into 10 elliptic curve factors and it is claimed that
$\gamma$ must act upon a particular one $E$ as $\pm 1$. From this, further analysis
leads to a contradiction. However, $E$ has $j$-invariant $0$ and it isn't
clear from the construction why $\gamma$ (which has order $2$ or $3$)  could not act
as a 3rd root
of unity on $E$. I tried to derive a different contradiction assuming this, but
everything seemed consistent if $\gamma$ was assumed to have order 3 and act on
each of the 6 CM components of $\JN$ as an appropriate 3rd root of unity.

$B_0(108)$ has order 108 by \cite{Bar08}. Kenku and Momose show
that all automorphisms of $\X$ are defined over the field they denote by $\kd$,
which is the ring class field mod 6 of $\kq = \Q(\sqrt{-3})$ : explicitly
$\kd$ is $\Q(\sqrt{-3},\sqrt[3]{2})$. The primes $p$ splitting completely in
$\kd$ are $31, 43, 109, \ldots$.

Out of interest, I decided to compute the automorphisms over the finite
fields $\F_p$ for the first few split primes using \magma \cite{Mag97}.
Equations for a
canonical model of $\X$ are found directly and reduced mod $p$ using
the modular form machinery. The full set of automorphisms is then 
returned very quickly by the built-in functions provided by Florian Hess.
To my surprise, in each case the number of automorphisms was 216, twice the
order of $B_0(108)$!

A slightly longer \magma\ computation over $\F_{31}$ returned an abstract group 
$G$ representing the automorphism group and it was readily verified by
further \magma\ function calls that $G$ did contain a subgroup $H$ of index 2
with the structure of $B_0(108)$ as described by Bars.
It remained to construct a new $u$ in characteristic 0. It is possible to just
mechanically run the \magma\ routines again but working over $\kd$ for
the automorphism group computations is {\it much} slower and besides, as an important
special case, it is desirable to provide a construction with some level of
transparent mathematical detail rather than just the output of a generic
computer program.
\medskip 

In the third section, I review Kenku and Momose's analysis of the
$N=108$ case and show that, after removing the error, it can be adapted
to prove that $B_0(108)$ is of index 1 or 2 in $A_0(108)$ and also give the
isomorphism type of $A_0(108)$ in the index 2 case.

In the fourth section, I give my construction of a new automorphism $u$ of order 2,
using detailed modular information about $\X$ and its Jacobian. Specifically,
I use the explicit action of standard generators of $B_0(108)$ on a natural
modular form basis for the differentials of $\X$ along with the commutator relations
within $A_0(108)$ for $u$ to find a fairly simple matrix, involving 2
undetermined parameters $a$ and $b$, giving the action of $u$ on the differentials.
To proceed further, I used computer computations
to first determine the relations for the canonical model of $\X$ w.r.t.~the
differential basis and then to solve for $a$ and $b$. The latter involves
computing a Gr\"obner basis
for the zero-dimensional ideal in $a$ and $b$ that comes from substituting the
matrix for $u$ into the canonical relations. The resulting
automorphism on the canonical model is defined over $\kd$ but not $\kq$ (the
field of definition of $B_0(108)$), as it should be.
\medskip

Subsequent to my discovery and semi-computerised construction of a new
automorphism, Elkies learned of the error through Mark Watkins. Using a neat
function-theoretic argument, similar in some respects to the $N=63$ case, he
was able to derive a particularly simple geometric model of $\X$ as the intersection of two
cubics in $\Prj^3$ as well as explicitly writing down all automorphisms without the
need for computer computations. Elkies construction gives an independent verification of the
corrected result for $N = 108$ and will be published by him elsewhere.
\medskip

Finally, I believe that there are no other exceptional $N$. \cite{KM88} is a very nice paper
but it does seem to contain a number of mistakes, most of which have no bearing on the final
result. In particular

\begin{itemize}
\item The last two numbers in the statement of Lemma 1.6 should be $2^3\cdot 3^2$
and $2^2\cdot 3^3$
\item A number of expressions in the proof of Lemma 1.6 are incorrect. Especially, most of
the expressions in the table for $\mu(D,p)$ in different cases are wrong.
\item $N=216=2^33^3$ listed in Corollary 1.11 and treated as a special case in the rest of the
paper along with the other 3 values is not actually a special case! This is clearly harmless.
\item In Lemma 2.15, $p^2-1$ should be $p^2+1$.
\item In the first part of the main theorem, 2.17, the list of $N$ for various $l \le 11$
that cannot be eliminated by lemmas 2.14 and 2.15 contain a number of cases that can be
easily eliminated by Corollary 2.11. However, there are a number of unlisted cases that
cannot be eliminated by any of these results. If I have worked it out correctly, I think that
these are (I'm ignoring the obvious typos in the lists here - e.g. the second $N$ listed for
$l=11$ should have $2^2$ rather than $2^3$ as a factor) $2^2\cdot 3^2 \cdot 7$ for
$l=5$, $2^4\cdot 11$ for $l=3$ and $3^3\cdot 5$ and $3^2\cdot 19$ for $l=2$.
The first two can be eliminated using Lemma 2.16 as is done for the other listed values of
$N$. For the other two ($l=2$) cases, a slightly improved version of lemma 2.15 can be applied 
over $p=2$ with $D = D_2$ that takes into account multiplicities in both the zeroes and poles
of $D$. This gives an upper bound 26 for $\mathcal{X}_0(N)(\F_4)$ which is less than the actual 
values (30 and 28 respectively) for the two cases.
\end{itemize}

The special analyses for the final six cases (ignoring $N = 37$ and $63$) in the proof of 
Theorem 2.17 all seem OK to me except when $N=108$. This is borne out by some \magma\
computations I performed more recently that are described in the next section.

\section{Computer verifications for $N \le 200$}

I performed a batch of computer computations using \magma\ to check the size of
the automorphism group of $X_0(N)$ over an appropriate small finite field $\F_p$
for all $N \le 100$ for which the genus is $\ge 2$ and for all non-square-free $N$
with $101 \le N \le 200$. The square-free case is covered by Ogg's earlier result.
In fact, $N = 4M$ or $8M$ with $M$ odd and square-free is also a simpler case
where all cusps are rational, there are no no non-cuspidal rational points (by
results of Kenku and Mazur : references 7-10 and 15 of \cite{KM88}) and all
automorphisms are defined over
$\Q$ ($k(N) = \Q$ in Kenku and Momose's notation). The result then follows from
Corollary 2.3 of \cite{KM88}, which shows that the subgroup of automorphisms
fixing $i\infty$ is of order 2 (generated by $\left({1 \atop 0}{1/2\atop 1}\right)$),
and the fact that $\#B_0(N)$ is twice the number of cusps ({\it cf} Proposition 2.8,
\cite{KM88}). Alternatively, there are no elliptic points when $4 | N$, so if an
automophism $\alpha$ preserves the cusps, its restriction to $Y_0(N)$ lifts to
an automorphism
of the upper half-plane, which is the simply-connected covering surface of $Y_0(N)$.
It then follows easily that $\alpha \in B_0(N)$. We checked these cases anyway. 
\medskip

The first section of \cite{KM88} is largely concerned with determining a number
field over which all automorphisms are defined. If $k(N)$ is the composite of
all quadratic fields the square of whose discriminant divides $N$, they show that
$k(N)$ will do except possibly in 3 cases: $N = 108, 256$ and $512$ ($216$ is also erroneously
listed) when the $k'(N)$ defined in Remark 1.12 can be used.
The primes $p$ were chosen to not divide $N$ and to split in $k(N)$ (or $k'(N)$).
The full $A_0(N)$ then injects into the group of automorphisms over $\F_p$ and
computing these gives an upper bound for $\#A_0(N)$.
\bigskip

Since the field of definition is of vital importance for the validity of the
computations as well as the theoretical results, we first performed some computations
to check the result of Lemma 1.6 \cite{KM88},
where there are errors in the list of exceptional $N$ and in the
formulae for $\mu(D,p)$, as noted in the introduction. Specifically, the formulae
on pages 58 and 59 for $\mu(D,p)$ in the cases $p \nmid D$, $p$ and $n$ odd or even, are
all too large by $n-1$ and in the formulae for $p=2$, $4 \parallel D$, $-4$ should be
replaced by $-6$.

We note that formula (1.8) on page 57 of \cite{KM88}, which gives an upper bound for
$g_C(N)$, actually gives an equality after a slight adjustment when $9|N$.
In that case, the term
on the right of (1.8) for $D=3$, $(1/3)\prod_{p|N}\mu(3,p)$, must be replaced by
$(1/3)[(\prod_{p|N}\mu(3,p))-2^n]$, where $n \ge 0$ is the number of distinct primes
apart from $3$ that divide $N$. The reason is as follows. The product would give the 
appropriate contribution for Hecke characters $\chi$ of $\Q(\zeta_3)$ except that it
doesn't take into account the extra condition on the `finite part', $\lambda$, of the 
character
($\chi((\alpha)) = \lambda(\alpha)\alpha$) that $\lambda(\zeta_3) = \zeta_3^{-1}$
where $\zeta_3$ is a primitive third root of unity. 
There are an equal number of $\lambda$ satisfying
$\lambda(\zeta_3) =\zeta_3$, as is easily seen by considering
complex conjugates of characters, but there is a slightly greater number that satisfy
$\lambda(\zeta_3) = 1$. This accounts for the negative adjustment before dividing by 3.
The exact value, $-2^n$, can be derived by showing it to be $-1$ when $N=3^r$ and then
arguing by induction.
\medskip

Given the formula for $g_C(N)$ in terms of $\mu(D,p)$ and the corrected formulae
for the latter, we wrote a short program to compute $g_C(N)$. We performed a
partial check on the correctness of our formulae by comparing our computed values of
$g_C(N)$ for $N \le 1000$ with the dimension of the subspace of weight 2 cusp forms
with complex multiplication computed via William Stein's modular forms package.
Then we computed $g_C(N)$ and $g_0(N)$ (the full genus of $X_0(N)$) by formula
for all $N \le 10000$ (which took less than a second) and found that
the only $N$ for which $2 \le g_0(N) \le 1+g_C(N)$ are indeed those listed in
typo-corrected Lemma 1.6 (minus $216$): namely, 54, 72, 81, 108, $2^6, 2^7, 2^8, 2^9$.
$g_0(216)=25$ and $g_C(216)=10$. The bounds (even better using the smaller values
of $\mu(D,p)$) given by Kenku and Momose to restrict the set of
$N$ for which $g_0(N) \le 1+g_C(N)$ show that all such bad $N$ are comfortably less
than $10000$.

We also checked with Stein's package that when $N = 54,72,81,64$ or $128$, $J_C(N)$
decomposes up to isogeny into a product of elliptic curves which all have CM by
the quadratic field $k(N)$ and are isogenous over $k(N)$. Hence, all endomorphisms of
$J_C(N)$ are defined over $k(N)$ in these cases. This leaves the 3 values
$108, 2^8, 2^9$ given in Corollary 1.11.
\bigskip

To compute the number of automorphisms over $\F_p$, $p \nmid N$, we needed to
compute equations
for the reduction of the $X_0(N)$. In the hyperelliptic cases, we simply took
well-known minimal Weierstrass models with integer coefficients and reduced these
mod $p$. In the other cases (with one exception), we computed a mod $p$ canonical
model. The structure of $\mathcal{X}_0(N)$, the minimal model of the $X_0(N)$ 
(see \cite{KM85}), and Katz'
theory on the relation between weight 2 cusp forms and differentials
in the integral case shows that if $f(q)$ is a weight 2 form with integer $q$-expansion,
then $\omega_f = f(q)(dq/q)$ extends to a holomorphic differential on the smooth part of
$\mathcal{X}_0(N)$ and that $q$ extends to a formal parameter for the $\Z$-section
$\boldsymbol{\infty}$ of $\mathcal{X}_0(N)$ with generic point $i\infty$ (the formal
completion of the section is the formal spectrum of $\Z[[q]]$). This shows that
the reduction mod $p$ of the power series $q$-expansion of $f$ divided by $q$
gives the formal power series expansion of  $(\omega_f \mbox{ mod } p)/dq$
at the mod $p$ reduction of $i\infty$ with respect to local parameter $q$.

Therefore, if
$f_1,\ldots,f_g$ is a basis for the weight 2 cusp forms (over $\C$) with
integer $q$-expansions and if the reductions mod $p$, $\tilde{F}_1(q),\ldots,
\tilde{F}_g(q)$ of these expansions divided by $q$ remain linearly-independent
over $\F_p$, then the
$\tilde{F}_i(q)$ are the local expansions of a mod $p$ basis of differentials.
Hence the polynomial relations between them over $\F_p$ are the defining equations for the
mod $p$ canonical image. Just as over $\C$, to verify a homogeneous degree $d$
polynomial relation, we only need to check it up to $O(q^{2(g_0(N)-1)d+2})$ since a
product of $d$ of the forms gives a holomorphic $d$-differential on the reduction
of $X_0(N)$.

To compute a basis for degree $d$ canonical relations, we computed the kernel
over $\F_p$ of the matrix with rows containing the coefficients of the
$q$-expansions up to sufficient degree of degree $d$ monomials in the $\tilde{F}_i(q)$.
The theory of canonical models of curves tells us
that for genus $\ge 5$, the ideal of relations of the canonical image is generated by 
quadrics and
possibly cubics (in trigonal or plane quintic genus 6 cases) and that the
canonical image is isomorphic to the curve unless the curve is hyperelliptic,
when it is a rational normal curve. For genus $3$ and $4$, things are also
straightforward.

It is just possible that a mod $p$ reduction of $X_0(N)$ could degenerate into
a hyperelliptic curve. However, this did not occur in any case, as was easily checked
by computing the degree of the canonical image to be $2g_0(N)-2$ rather than
$g_0(N)-1$. Also, no trigonal or plane quintic cases occured, as was checked by
computing that the quadric relations defined a dimension $1$ (not $2$) scheme.
\bigskip

After computing mod $p$ equations, the next stage was to compute a \magma\ 
function field for the mod $p$ curve. The standard \magma\ function for this was
used, which constructs it from a well-chosen
affine coordinate ring by localising at an appropriate base variable and then computing
a diagonal Gr\"obner basis for the zero-dimensional localised ideal of relations.
For $N \le 100$, this never took longer than a few seconds. For the higher genus
cases in the $101 \le N \le 200$ range, the computations took longer and 
Gr\"obner sensitivity to the defining equations as well as the degree of the function field
over its chosen $\F_p(x)$ base became more important. A number of slight variations
were used, switching to another if processing was going slowly. In each case, we
started with an basis of weight 2 forms with integer coefficients, computed using
Stein's modular forms package. We then applied a $\Z$-linear transformation to this, 
corresponding to taking an LLL lattice-basis for the matrix of leading coefficients.
This idea is used by Mark Watkins in his implementation of the general \magma\
function to compute canonical models of $X_0(N)$ and leads to $q$-expansions
with smaller coefficients and nicer canonical relations. In the default version,
we just computed the quadric relations of the mod $p$ reduction of this basis
for the canonical model. As a variant, which was much faster in some cases, we
echelonised the mod $p$ matrix of coefficients of this basis. This generally
produces less-sparse canonical defining polynomials, but can also produce a smaller
degree rational function given by the ratio of two homogeneous canonical coordinates.
This is picked up by the standard function, which chooses
them in order to construct a smaller degree function field. In a few bad cases,
both methods were very slow, but the third method of getting canonical equations
by reducing a further LLL-ised set of integral canonical relations computed for
the basis over $\Q$ worked. In the end, using one of these methods we computed
the mod $p$ function fields in no more than a few minutes in the worst cases, except
for $N=120$. For this one case, we computed the function field
as a biquadratic extension of that of the hyperelliptic curve $X_0(30)$,
first extending to $X_0(60)$ and then to $X_0(120)$ by adding smallish degree
functions given by eta-products and using \magma's divisor machinery for curves
to help compute the quadratic relation that each new function satisfied over
the previous function field.
\medskip

Finally, we computed all automorphisms of $X_0(N)$ over $\F_p$
by passing the mod $p$ function field to the buit-in function for this which
was provided by Florian Hess. This basic method, without further simplifications,
was adequate in most cases, although it took 2-3 hours to return the result
in some of the worst. In the highest genus cases where $g_0(N) \ge 22$
($N = 156,168,180,188,198$), we employed the following speed-up. In each case,
\magma\ quickly computed that the number of
$\F_p$-rational places of the function field was small - equal to the number
of cusps - for the chosen small $p$. We then just used Hess' function to
compute the automorphisms that fixed a particular place. The number of these
multiplied by the number of places gave an upper bound for the
total number of $\F_p$-automorphisms. This was actually the exact number,
since the $\F_p$-places correspond precisely to the reduction of the
cusps and $B_0(N)$ acts transitively on the cusps in each of these cases.

In this way, we computed an upper bound for $\#A_0(N)$ for each $N$ and found
that it was exactly $\#B_0(N)$ except in the three exceptional cases of
$N=37,63,108$, when it was twice $\#B_0(N)$. In only one case, $N=132$,
did our first choice of $p$ ($p=5$) not lead to the desired bound.
In that case, recomputing with $p=7$ worked. Our range of $N$ included
the 7 values mentioned in the introduction for which Kenku and
Momose had to perform a special case analysis.

Finally, we should note that the mod $p$ computations were much faster than
the corresponding characteristic $0$ ones (over the various $k(N)$) would
have been. Beside the problem of coefficient blow-up, Hess' function
needs to compute and work with the Weierstrass places of the function field,
at least in char. 0. In many cases, there are only a few of these places,
corresponding to many Galois-conjugate points, which are of large residue
degree. A few of the internal computations involve extending to these
residue fields, which causes processing to grind to a halt due to the
overhead of working over number-fields of very high degree. Over finite
fields, particularly small ones, the Weierstrass places split into a
larger number of small degree ones and working over finite extensions
of $\F_p$ is not so computationally onerous in any case. Also, there is
the choice to use $\F_p$-rational places, rather than Weierstrass places.

\section{Automorphisms of $\X$: Generalities}

We reconsider the analysis of the $\X$ case as given on pages 72 and 73 of 
\cite{KM88} and show that the correct conclusion is that $B_0(108)$ is of index
1 or 2 in $A_0(108)$ rather than that $A_0(108)$ is necessarily equal to 
$B_0(108)$.
\medskip

\noindent\underline{Notation:} 
\smallskip

$\X$, $J = J_0(108)$, $A_0(108)$, $B_0(108)$, $\kq$ and $\kd$ are as described
in the introduction. $\sigma$ denotes one of the two generators of
$G(\kd/\kq)$.

$w_n$ will denote the Atkin-Lehner involution on $\X$
for $n|108$, $(n,108/n)=1$ (see, eg, \cite{Mi89} or \cite{Bar08}). 
Explicitly, we could take matrix representatives mod $\R^*\Gamma_0(108)$
for the actions of $w_4, w_{27}$ and $w_{108}$ on the extended upper half-plane as

$$ w_4 = \left(\begin{array}{cc}28 & 1 \\108 & 4 \end{array}\right)\qquad
   w_{27} = \left(\begin{array}{cc}27 & -7 \\108 & -27 \end{array}\right)\qquad
   w_{108} = \left(\begin{array}{cc}0 & -1 \\108 & 0 \end{array}\right)$$

\noindent $w_1$ is trivial. Up to scalars, the matrix for $w_{27}$ is an involution and 
the matrix for $w_{108}$ is the product of those for $w_{27}$ and $w_4$.

For $v|6$, $S_v$ will denote the element of $B_0(108)$ represented by the
matrix $\left({1\atop 0}{(1/v)\atop 1}\right)$.
\medskip

$B_0(108)$ is generated by $w_4, w_{27}, S_2$ and $S_3$. Its group structure
is described fully later in the next section. We note here that the first three
generators have order 2 and the last has order 3. The subgroup $S$ of $B_0(108)$
commuting with $\langle w_4, w_{27}\rangle$ is

$$ \langle w_4\rangle \times \langle w_{27}\rangle \times \langle\tau_3\rangle$$

\noindent where $\tau_3$ is the element of order 3 in the centre of $B_0(108)$
defined by

$$ \tau_3 := S_3 w_{27} S_3 w_{27} $$

\noindent ($S_3$ and $w_{27} S_3 w_{27}$ commute).
\medskip

The argument on page 73 of \cite{KM88} considers a hypothetical automorphism
$u$ in $A_0(108)$ not lying in $B_0(108)$. It is shown that $u$ is defined
over $\kd$ but not over $\kq$ and the non-trivial automorphism $\gamma$
is defined as $u^\sigma u^{-1}$. Note that all cusps are defined over $\kq$,
that $B_0(108)$ acts transitively on the cusps and that, by their Corollary 2.3,
any automorphism is determined by its images of $\infty$ and any other cusp.
This shows, in particular, that all elements of $B_0(108)$ are defined over
$\kq$. $\gamma$ is shown to lie in $B_0(108)$.
\smallskip

Let $f_{27}$, $f_{36}$ and $f_{108}$ denote the primitive cusp forms associated
to the unique isogeny classes of elliptic curves with conductors 27, 36 and 108
respectively. These curves all have complex multiplication by orders of $\kq$.
Kenku and Momose consider the decomposition up to isogeny of $J$ into the product
$J_H \times J_{C_1} \times J_{C_2}$ where $J_H$ is the part without CM,
$J_{C_1}$ is associated to the eigenforms $\{f_{36}(z), f_{36}(3z), f_{108}(z)\}$
and $J_{C_2}$ is associated to the eigenforms $\{f_{27}(z), f_{27}(2z), f_{27}(4z)\}$.
They show that $\gamma$ acts trivially on the $J_H$ factor and that its order $d$
and the genus $g_Y$ of the quotient $\X/\langle\gamma\rangle$ satisfies (i) $d=2$,
$g_Y = 4,5$
or (ii) $d=3$, $g_Y=4$. It is also shown that $\gamma$ commutes with $w_4$ and
$w_{27}$ and so lies in $S$.

$E$ is the new elliptic curve factor of $J_{C_1}$ corresponding to $f_{108}$.
The error comes with the line ``Then $\gamma$ acts on $E$ under $\pm 1$". This
eliminates case (ii) above and leads to a contradiction on the existence of $u$.
However, there is the possibility that 
\smallskip

\noindent (*) $\gamma$ acts on (the optimal quotient isogeny 
class of) $E$ by a non-trivial 3rd root of unity and case (ii) occurs.
\smallskip 

We see in the next section that this actually can occur when we explicitly 
construct such a $u$. To have order 3 and lie in $S$, $\gamma$ must equal $\tau_3$ 
or $\tau_3^{-1}$. That (*) holds for such $\gamma$ follows from the determination
of the action of the generators of $B_0(108)$ on a nice basis for the cusp forms
given in the construction. This shows that
$\tau_3$ fixes the non-CM forms defining the $J_H$ factor
and multiplies each of the six CM Hecke eigenforms given above by some
non-trivial 3rd root of unity as required.
\medskip

So for any automorphism $u$, $u^\sigma u^{-1}$ is trivial or equal to $\tau_3$
or $\tau_3^{-1}$. As Kenku and Momose show that all automorphisms defined over $\kq$ 
are in $B_0(108)$, this implies that $B_0(108)$ is of index at most three in 
$A_0(108)$.

However, this can be improved by considering the action on the reduction mod 31
of $\X$ and arguing as Kenku and Momose do to show that $\gamma$ is defined over $\kq$. All 
automorphisms are defined over $\F_{31}$ as 31 splits in $\kd$. If $u$ and
$v$ are two automorphisms not in $B_0(108)$, then exactly the same argument
near the top of page 73 applied to $u$ and $u^\sigma$ can be applied to $u$
and $v$ to show that $vu^{-1}$ lies in $B_0(108)$. This shows that
$B_0(108)$ is of index at most 2 in $A_0(108)$. Note that
the sentence on page 73 starting ``Applying lemma 2.16 to $p = 7\ldots$'' should
contain $p = 31$ rather than $p = 7$ and there should be a comment that lemma 2.16
is being applied here with any pair of cusps replacing $\mathbf{0}$ and $\infty$,
which is permissible as the same proof works. So, replacing $\sigma$ by $\sigma^{-1}$
if necessary, we have that (remembering that $\tau_3$ is in the centre of $B_0(108)$)
\medskip

\noindent (+) $B_0(108)$ is of index 1 or 2 in $A_0(108)$ and any automorphism 
$u\notin B_0(108)$ satisfies $u^\sigma u^{-1} = \tau_3$.
\bigskip

Now, we assume that $A_0(108)$ is bigger than $B_0(108)$ and show that its group structure
can then be determined from the above information and the abstract group structure of 
$B_0(108)$.
\medskip

We denote a cyclic group of order $n$ by $C_n$.
\cite{Bar08} gives the structure of $B_0(108)$. Abstractly, it is the
direct product $D_6 \times (C_3 \wr C_2)$ where the first factor is the
dihedral group of order 6 and the second is the order 18 wreath product
(the semidirect product of $C_3 \times C_3$ by $C_2$, the generator of $C_2$
 swapping the two $C_3$ factors). The $D_6$ factor is generated by $S_2$ and 
$w_4$, which both have order 2. The wreath product is generated by $S_3$ (order 3)
and $w_{27}$ (order 2), so that $S_3$ and $w_{27}S_3w_{27}$ are two commuting
elements of order 3 generating the order 9 subgroup. The centre of $B_0(108)$
is of order 3, generated by $\tau_3 = S_3w_{27}S_3w_{27}$.

The automorphism group of $B_0(108)$ is easy to determine on writing it as 
$D_6 \times D_6 \times C_3$. The outer automorphism group is
$C_2 \times C_2$. This can also be easily checked in \magma, for example.
\medskip

Now if $u$ is not in $B_0(108)$, $u^2$ is in $B_0(108)$ and so is fixed by $\sigma$.
Then, (+) above shows that $u\tau_3u^{-1} = \tau_3^{-1}$ so that $u$ acts by conjugation
on $B_0(108)$ (which is normal in $A_0(108)$ having index 2) as an outer automorphism,
since $\tau_3$ is central in $B_0(108)$. Also, we can assume $u$ has 2-power order and,
as the kernel of the map of $B_0(108)$ to its inner automorphism group is of
order 3, $A_0(108)$ is then determined up to isomorphism if we can determine
the image of $u$ in the outer automorphism group $H$ of $B_0(108)$. $H$ has 3
non-trivial elements giving extensions of $B_0(108)$ of degree 2. But the
condition that $u$ doesn't centralise $\tau_3$ excludes one of these elements.
Another element would leads to a $u$ of order 2 commuting with $\langle w_4,w_{27}\rangle$
and $S_2$. From the explicit action of $S_2$ on weight 2 forms (see next section)
we see that $u$ would have to preserve the $(w_{27}-1)J_{C_1}$ ($=E$) and
$(w_4-1)J_{C_2}$ elliptic curve factors of the Jacobian, so act as $\pm 1$ on $E$.
$\gamma$ {\it would} then act trivially on $E$ and the argument of Kenku and Momose
would properly lead to a contradiction. Thus, there is only one possibility for
$u$ in $H$ and one possible group structure for $A_0(108)$. Explicitly, we find

\begin{lemma}\label{le1}
If $A_0(108)$ is larger than $B_0(108)$, then it contains $B_0(108)$ as a subgroup
of index 2 and is generated by $B_0(108)$ and an element $u$ of order 2 that acts on
$B_0(108)$ by conjugacy as follows:
$$ uw_4u = w_{27} \qquad uw_{27}u = w_4 $$
$$ uS_2u = S_3w_{27}S_3^{-1} = S_3^{-1}\tau_3^{-1}w_{27} \qquad
     uS_3u = S_2w_4\tau_3 $$

\noindent For an appropriate choice of $\sigma$, $u^\sigma u^{-1} = \tau_3$.
\end{lemma}

\section{Construction of a new automorphism}

The notation introduced at the start of the last section is still in force.

\noindent {\bf Conventions:}
\smallskip

If $u$ is an automorphism of $\X$, then we also
think of it as an automorphism of $J$ by the ``Albanese" action: a degree zero
divisor $\sum_i a_i P_i \mapsto \sum_i a_i u(P_i)$. If $\X$ is embedded in $J$ in
the usual way by $i: P \mapsto (P) - (\infty)$ then the actions are compatible up to
translation by $(u(\infty))-(\infty)$. As global differentials on $J$ are translation
invariant, this means that the pullback action $u^*$ on global differentials of
$J$ or $\X$ is the same if we identify global differentials of $J$ and $\X$ by
the pullback $i^*$.

When we identify the weight 2 cusp form $f(z)$ of $\Gamma_0(108)$ with the 
complex differential $(1/2\pi i)f(z)dz$ on $\X$, if $u \in B_0(108)$ then
$u^*f$ is $f|_2 u$ in the notation of Sec. 2.1 \cite{Mi89}, identifying $u$ with
the 2x2 matrix representing it. As we only deal with weight 2 forms we omit the
subscript 2.

If we say that $u$ is represented by matrix $M$ w.r.t. a basis $f_1,\ldots,f_n$
of cusp forms/differentials, we mean that $u^*f_i = \sum_jM_{ji}f_j$. So if
$u$ and $v$ are represented by $M$ and $N$, then $uv$ is represented by $MN$.
\bigskip

$J_0(108)$ decomposes up to isogeny into a product of 10 elliptic curves defined
over $\Q$, as partially described in \cite{KM88}. We work with a natural basis for 
the weight 2 cusp forms of level 108 coming from multiples of the primitive
forms $f_{27}$, $f_{36}$ and $f_{108}$ which generate the CM part as in the last
section, and the two primitive level 54 forms $f^{(1)}_{54}$, $f^{(2)}_{54}$ and their 
multiples by 2 which generate a 4-dimensional non-CM complement. All the forms have 
rational $q$-expansions. We give the initial $q$-expansions of the primitive forms,
isogeny classes of elliptic curves over $\Q$ that they correspond to and the 
eigenvalues for the Atkin-Lehner involutions of the base level (which we refer to
as $W_n$ to differentiate from the $w_n$ involutions for level 108).
\medskip

\noindent\underline{Conductor 27} 
\smallskip

$$ f_{27} = q - 2q^4 - q^7 + 5q^{13} + 4q^{16} - 7q^{19} +O(q^{25}) $$

$$ W_{27} = -1 $$

$$ E_{27}: y^2+y=x^3 \simeq y^2=x^3+16 $$

\noindent $X_0(27)$ is of genus 1 and this is a well-known case (see \cite{Li75}).
$f_{27}$ is $\{\eta(3z)\eta(9z)\}^2$ where $\eta$ is the Dedekind eta function
(\S 4.4 \cite{Mi89}).
\medskip

\noindent\underline{Conductor 36} 
\smallskip

$$ f_{36} = q - 4q^7 + 2q^{13} + 8q^{19} + O(q^{25}) $$

$$ W_9 = 1, W_4 = -1 $$

$$ E_{36}: y^2=x^3+1 $$

\noindent $X_0(36)$ is of genus 1 and this is a well-known case (see \cite{Li75}).
$f_{36}$ is $\eta(6z)^4$.
\medskip

\noindent\underline{Conductor 108} 
\smallskip

$$ f_{108} = q + 5q^7 - 7q^{13} - q^{19} + O(q^{25}) $$

$$ w_{27} = 1, w_4 = -1 $$

$$ E_{108}: y^2=x^3+4 $$

\noindent $f_{108}$ and the action of the Atkin-Lehner operators come from Tables 3
and 5 of \cite{MF75} or from a modular form computer package such as William Stein's.
It is easy to check that the CM elliptic curve $E_{108}$ has conductor 108 with
$f_{108}$ as its associated modular form (which is of the type described in Thm 4.8.2
of \cite{Mi89} with $K = \kq$).
\medskip

\noindent\underline{Conductor 54} 
\smallskip

$$ f^{(1)}_{54} = q - q^2 + q^4 + 3q^5 -q^7 - q^8 - 3q^{10} - 3q^{11} +O(q^{12}) $$

$$ w_{27} = -1, w_2 = 1 $$

$$ f^{(2)}_{54} = q + q^2 + q^4 - 3q^5 -q^7 + q^8 - 3q^{10} + 3q^{11} +O(q^{12}) $$

$$ w_{27} = 1, w_2 = -1 $$

$$ E^{(1)}_{54}: y^2+xy=x^3-x^2+12x+8\quad E^{(2)}_{54}: y^2+xy+y=x^3-x^2+x-1 $$

\noindent The $f^{(i)}_{54}$ and the action of the Atkin-Lehner operators again come from Tables 3
and 5 of \cite{MF75} or from a modular form computer package.
It is easy to check that the elliptic curves given have conductor 54 and have the
respective $f^{(i)}_{54}$ as their associated modular forms. In fact the $E$s are quadratic
twists of each other by $-3$ and the two $f^{(i)}_{54}$ are twists by the quadratic character
of $\kq$. 
\medskip

\begin{definition}\label{defndelta}
$\delta_n$ is the operator on modular forms given by the matrix
$\left({n\atop 0}{0\atop 1}\right)$, so that if $f(z)$ is a weight
2 form, $(f|\delta_n)(z)$ is the form $nf(nz)$.
\end{definition}
\medskip

\begin{definition}\label{basis}
$e_1,\ldots,e_{10}$ are the basis for the weight 2 cusp forms of $\Gamma_0(108)$
defined as follows:
$$ e_1 = f^{(2)}_{54} - f^{(2)}_{54}|\delta_2 \qquad
                                e_2 = f^{(2)}_{54} + f^{(2)}_{54}|\delta_2 \qquad
 e_3 = f^{(1)}_{54} - f^{(1)}_{54}|\delta_2 \qquad
                                e_4 = f^{(1)}_{54} + f^{(1)}_{54}|\delta_2 $$
$$ e_5 = f_{27} + f_{27}|\delta_4 \qquad e_6 = f_{27}|\delta_2 \qquad
           e_7 = f_{36} + f_{36}|\delta_3 $$
$$ e_8 = f_{108}\qquad e_9 = f_{27} - f_{27}|\delta_4 \qquad 
           e_{10} = f_{36} - f_{36}|\delta_3 $$
The standard decomposition into new and old forms (Section 4.6, \cite{Mi89}) shows
that $e_1,\ldots,e_{10}$ form a basis for the weight 2 cusp forms of $\Gamma_0(108)$

Identifying these cusp forms with differential forms on $\X$ and $J$,
 $V := \langle e_1,\ldots,e_4\rangle$ is the subspace corresponding to differentials of $J_H$
and $W := \langle e_5,\ldots,e_{10}\rangle$ the subspace corresponding to
$J_{C_1} + J_{C_2}$.
All endomorphisms of $J$ preserve these subspaces.   
\end{definition}

\begin{lemma}\label{le2}
With respect to the basis $e_1,\ldots,e_4$ of $V$, $w_{4},w_{27},S_2$ and $S_3$
act by the following matrices ( $\zeta := exp(2\pi i/3)$,
$\rtm := \zeta-\zeta^{-1}$ )

$$\left(\begin{array}{cccc}1 & 0 & 0 & 0\\0 & -1 & 0 & 0\\0 & 0 & 1 & 0\\
    0 & 0 & 0 & -1\end{array}\right) \qquad
  \left(\begin{array}{cccc}1 & 0 & 0 & 0\\0 & 1 & 0 & 0\\0 & 0 & -1 & 0\\
    0 & 0 & 0 & -1\end{array}\right)$$

$$ \frac{1}{2}\left(\begin{array}{cccc}-1 & -3 & 0 & 0\\-1 & 1 & 0 & 0\\0 & 0
     & -1 & -3\\0 & 0 & -1 & 1\end{array}\right)\qquad
   \frac{1}{2}\left(\begin{array}{cccc}-1 & 0 & \rtm & 0\\0 & -1 & 0 & \rtm\\
     \rtm & 0 & -1 & 0\\ 0 & \rtm & 0 & -1\end{array}\right)$$

With respect to the basis $e_5,\ldots,e_{10}$ of $W$,$w_{4},w_{27},S_2$ and $S_3$
act by the following matrices

$$\left(\begin{array}{cccccc}1 & 0 & 0 & 0 & 0 & 0\\0 & 1 & 0 & 0 & 0 & 0\\
  0 & 0 & -1 & 0 & 0 & 0\\0 & 0 & 0 & -1 & 0 & 0\\0 & 0 & 0 & 0 & -1 & 0\\
  0 & 0 & 0 & 0 & 0 & -1\end{array}\right) \qquad
  \left(\begin{array}{cccccc}-1 & 0 & 0 & 0 & 0 & 0\\0 & -1 & 0 & 0 & 0 & 0\\
  0 & 0 & 1 & 0 & 0 & 0\\0 & 0 & 0 & 1 & 0 & 0\\0 & 0 & 0 & 0 & -1 & 0\\
  0 & 0 & 0 & 0 & 0 & -1\end{array}\right)$$

$$\left(\begin{array}{cccccc} -1/2 & 0 & 0 & 0 & -3/2 & 0\\
   0 & 1 & 0 & 0 & 0 & 0\\0 & 0 & -1 & 0 & 0 & 0\\0 & 0 & 0 & -1 & 0 & 0\\
   -1/2 & 0 & 0 & 0 & 1/2 & 0\\0 & 0 & 0 & 0 & 0 & -1\end{array}\right) \qquad
  \left(\begin{array}{cccccc}\zeta & 0 & 0 & 0 & 0 & 0\\
   0 & \zeta^{-1} & 0 & 0 & 0 & 0\\0 & 0 & -(1/2)\zeta^{-1} & 0 & 0 & -(1-\zeta)/2\\
   0 & 0 & 0 & \zeta & 0 & 0\\0 & 0 & 0 & 0 & \zeta & 0\\
   0 & 0 & -(1-\zeta)/2 & 0 & 0 & -(1/2)\zeta^{-1}\end{array}\right)$$
\end{lemma}

\begin{proof} The proof is a straightforward computation using relations between
Atkin-Lehner involutions and the $\delta_i$ and congruence conditions on
the exponents of the non-zero terms of the $q$-expansions to find the $S_2$ and
$S_3$ actions.
\smallskip

For $S_2$: All forms $f|\delta_i$ where $i$ is 2 or 4 are clearly fixed by $S_2$.
Generally, considering $q$-expansions, if a form $f$ is an eigenvalue of the Hecke
operator $T_2$ at its even base level with eigenvalue $e$, then we see that
$f|S_2 = -f + ef|\delta_2$. Note that $f_{36}|\delta_3$ is still an eigenvector
of $T_2$ with eigenvalue 0. This leaves only $f_{27}$ to consider. As it is
killed by $T_2$, the definition of the $T_2$ action quickly leads to
$f_{27}|S_2 = -(f_{27}+f_{27}|\delta_4)$.
\smallskip

For $S_3$: $f_{27}, f_{36}$ and $f_{108}$ all
have $q$-expansions where all non-zero terms $a_nq^n$ have $n = 1$ mod 3.
This follows from the fact that they are eigenvalues of all Hecke operators
and that $a_p = 0$ if $p=2$ mod 3, $p > 2$, as the associated elliptic curves
have supersingular reduction at these primes so $p|a_p$ and $|a_p| < 2\surd p$.
$a_2$ and $a_3$ are clearly also zero. As $f^{(1)}_{54}$ and $f^{(2)}_{54}$ are
twists by the $(\frac{-3}{.})_2$ quadratic character and are killed by the $T_3$
operator, $f^{(1)}_{54}+f^{(2)}_{54} = \sum_{n = 1 (3)} a_nq^n$ and 
$f^{(1)}_{54}-f^{(2)}_{54} = \sum_{n = 2 (3)} b_nq^n$. From these facts,
the action of $S_3$ on the basis follows easily.
\smallskip

For $w_4$: Simple matrix computations show that $(f|\delta_2)|w_4 = f|W_2$
and $f|w_4=(f|W_2)|\delta_2$ for a level 54 form $f$.
Similarly, $f|w_4=f|\delta_4$, $(f|\delta_2)|w_4=f|\delta_2$ and
$(f|\delta_4)|w_4=f$ for level 27 forms; $f|w_4=f|W_4$ and $(f|\delta_3)|w_4
=(f|W_4)|\delta_3$ for level 36 forms. The full $w_4$ action follows.
\smallskip

For $w_{27}$: Again, simple matrix computations show that $f|w_{27} = f|W_{27}$
and $(f|\delta_2)|w_{27} = (f|W_{27})|\delta_2$ for level 54 forms;
$f|w_{27} = f|W_{27}$ and $(f|\delta_i)|w_{27} = (f|W_{27})|\delta_i$ ( $i=2$ or 4)
for level 27 forms; $f|w_{27} = (f|W_9)|\delta_3$ and $(f|\delta_3)|w_{27} = 
(f|W_9)$ for level 36 forms.The full $w_{27}$ action follows.
\end{proof}

{\it Note:} From the above lemma, we see that $\tau_3$ acts trivially on $V$
and multiplies each $e_i$ for $i \ge 5$ by $\zeta$ or $\zeta^{-1}$ as asserted
in the last section.
\bigskip

Using the commutator conditions for a new automorphism $u$ of order 2 as described
in Lemma \ref{le1}, it is now easy to show that $u$ acts on weight two forms by a matrix
$\pm M$ (w.r.t. the $e_i$ basis) with $M$ of the form

$$\left(\begin{array}{cc} {\begin{array}{cccc}1 & 0 & 0 & 0\\0 & 0 & z^{-1} & 0\\
   0 & z & 0 & 0\\0 & 0 & 0 & 1\end{array}} & \mbox{\Huge{\bf 0}}\\
  \mbox{\Huge{\bf 0}} & {\begin{array}{cccccc} 0 & 0 & -za & 0 & 0 & 0\\
  0 & 0 & 0 & b & 0 & 0\\ (-za)^{-1} & 0 & 0 & 0 & 0 & 0\\
  0 & b^{-1} & 0 & 0 & 0 & 0\\ 0 & 0 & 0 & 0 & 0 & a\\
  0 & 0 & 0 & 0 & a^{-1} & 0\end{array}}\end{array}\right)$$

with $a,b \in \C^*$ and $z = \rtm$ as defined in Lemma \ref{le2}.
\medskip

We can now complete the construction of $u$ on a canonical model of $\X$ with
detailed computations that can be carried out using a suitable computer
algebra system. The author performed these with \magma. There are 3 steps.

\begin{enumerate}
\item Compute a basis $R$ for the degree 2 canonical relations for $\X$ embedded
into $\Prj^9$ via the differential basis corresponding to the forms $e_i$. These
relations generate the full ideal defining $\X$ in $\Prj^9$.
\item Substitute the automorphism of $\Prj^9$ given by $M$ into $R$ treating $a$ and
$b$ as indeterminates. Clear powers of $a$ and $b$ from denominators. The condition
that each new degree 2 form must lie in the the span of $R$ gives a number of
polynomial relations on $a$ and $b$ that generate a zero dimensional ideal $I_{a,b}$
of $\kq[a,b]$.
\item Compute a lex Gr\"obner basis of $I_{a,b}$. From this, we can read off all
solutions for $a$ and $b$ such that $M$ gives an automorphism of $\X$.
\end{enumerate} 
\medskip

For the first step we need to find a basis for the linear relations between the 55
weight 4 cusp forms $e_ie_j$, $1 \leq i \leq j \leq 10$. Considering it as a regular
differential of degree 2 (see Section 2.3 \cite{Mi89} - note that there are no
elliptic points here), a weight 4 form for $\Gamma_0(108)$ that vanishes to order
at least 2 at each cusp is zero iff it has a $q$-expansion $\sum_{n \ge 2} a_nq^n$ with
$a_n = 0$, $\forall n \le 38$. So the computation reduces to finding a basis for
the kernel of a $55 \times 37$ matrix with integer entries. In practise, it is
good to work to a higher $q$-expansion precision than 38 and we actually did the
computation with the expansions up to $q^{150}$. This still only took
a fraction of a second in \magma. Applying an LLL-reduction to get a nice basis
for the relations, the result is that the canonical model for $\X$ in $\Prj^9$
with coordinates $x_i$ is defined by the ideal generated by the following 28
degree 2 polynomials:

$$ x_3x_4 + x_6x_9 - x_5x_{10},\quad x_1x_2 - x_6x_9 - x_5x_{10},\quad
    x_2x_6 - x_3x_7 + x_1x_{10},\quad x_4x_5 - x_1x_8 + x_6x_{10} $$
$$ x_1x_8 - x_3x_9 + x_6x_{10},\quad x_4x_6 + x_1x_7 - x_3x_{10},\quad
    x_4x_7 - 2x_8x_9 + x_2x_{10},\quad x_3x_7 - 2x_5x_8 + x_1x_{10} $$
$$ x_2x_5 - 2x_3x_8 + x_1x_9,\quad x_2x_5 - 2x_6x_7 - x_1x_9,\quad
    x_2x_4 + x_7x_9 - 2x_8x_{10},\quad x_1x_7 - 2x_5x_9 + x_3x_{10} $$
$$ x_2x_3 - x_5x_7 - 2x_6x_8,\quad x_2x_3 + x_1x_4 - 2x_5x_7,\quad
    x_7^2 - x_2x_8 - x_4x_9 + x_{10}^2,\quad x_1^2 - x_3^2 + 2x_5x_6 $$
$$ 3x_1x_5 - 2x_4x_8 - x_2x_9,\quad 3x_6^2 - x_7^2 + x_{10}^2,\quad
    3x_1x_3 - x_7x_9 - 2x_8x_{10},\quad 3x_1x_6 + x_4x_7 - x_2x_{10} $$
$$ 3x_3x_6 - x_2x_7 + x_4x_{10},\quad 3x_3x_5 - x_7^2 - x_2x_8 - x_{10}^2,\quad
    3x_1^2 - x_4^2 - 2x_9x_{10},\quad x_2^2 - 3x_3^2 + 2x_9x_{10} $$
$$     x_2^2 + x_4^2 - 4x_7x_8 + 2x_9x_{10},\quad 
  x_2x_7 - 4x_8^2 + 2x_9^2 + x_4x_{10},\quad x_4x_8 + x_2x_9 - 2x_7x_{10},\quad
    3x_5^2 - x_2x_7 - x_9^2 - x_4x_{10} $$
\smallskip

We are using the fact that $\X$ is not hyperelliptic \cite{Ogg74}. This follows from
the above anyway, since there would be 36 canonical quadric relations if it were.
However, it needs to be checked that $\X$ is not {\it trigonal} (having a
degree 3 rational function) when there would be independent degree 3 relations.
For this, it is only necessary, for example, to verify that the ideal defined by the 
above polynomials has the right Hilbert series. This was easily verified
in \magma\ which uses a standard Gr\"obner based algorithm \cite{BS92}.
\medskip

For the second step, we work over $K(a,b)$, $K = \Q(z)$, apply the substitution
$x_i \mapsto \sum_{1 \le j \le 10} M_{j,i}x_j$ to the above polynomials, and the
rest is strightforward linear algebra. Applying the Gr\"obner basis algorithm, we
find that $I_{a,b}$ is generated as an ideal by the two polynomials
$$ b - za^2,\qquad a^3 + (1/2) $$
which gives 3 possibilities for $u$ with $a$ any cube root of $-1/2$. We remark that 
if $u$ is one of these automorphisms then the other two are $u\tau_3$ and $u\tau_3^{-1}$
as expected. We also check that $u^\sigma u^{-1} = \tau_3$ if
$a^\sigma = exp(2\pi i/3) a$.
\medskip

\begin{theorem} $B_0(108)$ is of index two in $A_0(108)$, which has the structure
described in Lemma \ref{le1}. $u$ is given explicitly on the canonical model of
$\X$ with the above defining equations by
$$ [x_1:x_2:x_3:x_4:x_5:x_6:x_7:x_8:x_9:x_{10}] \mapsto $$
$$ [x_1:zx_3:(1/z)x_2:x_4:(c/z)x_7:(c^2/z)x_8:(z/c)x_5:(z/c^2)x_6:-cx_{10}:-(1/c)x_9] $$
where $c^3 = 2$ and $z = \sqrt{-3}$ with $\Im(z) > 0$.
\end{theorem}
\medskip

{\it Remarks:}
\begin{enumerate}
\item The action of $u$ on differentials is given by $\pm M$ where $M$ is the matrix
on page 8. $M$ has 1 (resp.~$-1$) as an eigenvalue of multiplicity 6 (resp.~4). Let
$Y = \X/\langle u\rangle$. If the action was by $M$, then the genus of $Y$, $g_Y$,
would be 6.
The Hurwitz formula would then give a value of $-2$ for the number of fixed points of $u$.
Thus $u$ acts on differentials by $-M$, $g_Y = 4$, and $u$ has two fixed points on $\X$.
\item On our canonical model of $\X$, the cusp $\infty$ is given by the point
$(1:1:1:1:1:0:1:1:1:1)$ and the generators of $B_0(108)$ act via the matrices given
in Lemma \ref{le2}. It is then easy to compute all of the cuspidal points as
the cusps form a single orbit under $B_0(108)$ and to verify that
$u(\{cusps\}) \cap \{cusps\} = \emptyset$. This is in accordance with Corollary 2.4 of
\cite{KM85} which says that $u$ would lie in $B_0(108)$ if it mapped a cusp to a cusp.
\end{enumerate}

\bibliographystyle{amsalpha}
\bibliography{mike.bib}
\end{document}